\documentclass[12pt]{amsart}

\usepackage{enumerate, amsthm, amssymb}

\newtheorem{theorem}{Theorem}
\newtheorem{lemma}[theorem]{Lemma}

\newcommand{\bC}{\mathbb{C}}
\newcommand{\bQ}{\mathbb{Q}}
\newcommand{\bN}{\mathbb{N}}
\newcommand{\bR}{\mathbb{R}}
\newcommand{\bZ}{\mathbb{Z}}

\newcommand{\et}{\quad\text{and}\quad}
\newcommand{\GL}{\mathrm{GL}}
\newcommand{\hlambda}{\hat\lambda}
\newcommand{\Qbar}{\bar{\bQ}}
\newcommand{\ua}{\mathbf{a}}
\newcommand{\ux}{\mathbf{x}}
\newcommand{\uy}{\mathbf{y}}

\numberwithin{equation}{section}

\begin{document}

\baselineskip=14.5pt

\title[A measure of transcendence]
{A measure of transcendence\\ for singular points on conics}
\author{Damien ROY}
\address{
   D\'epartement de Math\'ematiques\\
   Universit\'e d'Ottawa\\
   585 King Edward\\
   Ottawa, Ontario K1N 6N5, Canada}
\email[Damien Roy]{droy@uottawa.ca}
\thanks{Research partially supported by NSERC}


\begin{abstract}
A singular point on a plane conic defined over $\bQ$ is a 
transcendental point of the curve which admits very good 
rational approximations, uniformly in terms of the height.  
Extremal numbers and Sturmian continued fractions are 
abscissa of such points on the parabola $y=x^2$.  In this 
paper we provide a measure of transcendence for singular 
points on conics defined over $\bQ$ which, in these two 
cases, improves on the measure obtained
by Adamczewski et Bugeaud.  The main tool is a quantitative 
version of Schmidt subspace theorem due to Evertse.
\par
\medskip
\noindent
{\sc R\'esum\'e.}
Un point d'une conique d\'efinie sur $\bQ$ est dit singulier s'il
est transcendant et admet de tr\`es bonnes approximations 
rationnelles, uniform\'ement en termes de la hauteur. 
Les nombres extr\'emaux et les fractions continues sturmiennes
sont les abscisses de tels points sur la parabole $y=x^2$.  
Nous \'etablissons ici une mesure de transcendance de points
singuliers sur les coniques d\'efinies sur $\bQ$ qui, dans 
ces deux cas, am\'eliore la mesure obtenue
pr\'ec\'edemment par Adamczewski et Bugeaud.  L'outil principal
est une version quantitative du th\'eor\`eme du sous-espace
de Schmidt due \`a Evertse.
\end{abstract}

\maketitle

%
%

\section{Introduction}
\label{sec:intro}

In \cite[\S5.2]{AB}, Adamczewski and Bugeaud established a measure of transcendence for the extremal numbers from \cite{RcubicI} as well as for the Sturmian continued fractions from \cite{BL}.  The goal of this paper is to prove the following sharper measure which applies to a larger class of numbers.

\begin{theorem}
Let $(\xi,\eta)\in\bR^2$.  Suppose that $1,\xi,\eta$ are linearly independent over $\bQ$, and that $f(\xi,\eta)=0$ for some irreducible polynomial $f(x,y)\in\bQ[x,y]$ of degree $2$, not in $\bQ[x]$.  Suppose furthermore that there exists a real number $\lambda>1/2$ such that the inequalities
\begin{equation}
 \label{intro:thm:eq1}
 |x_0|\le X, \quad |x_0\xi-x_1|\le X^{-\lambda}, \quad |x_0\eta-x_2|\le X^{-\lambda}
\end{equation}
have a non-zero solution $(x_0,x_1,x_2)\in\bZ^3$ for each large enough real number $X\ge 1$.  Then $\xi$ is transcendental over $\bQ$ and there exists a computable constant $c>0$ such that, for each pair of integers $d\ge 3$ and $H\ge 2$ and each algebraic number $\alpha$ of degree $d(\alpha)\le d$ and naive height $H_0(\alpha)\le H$, we have
\begin{equation}
 \label{intro:thm:eq2}
 |\xi-\alpha|
   \ge H^{-w(d)} \quad \text{where} \quad w(d)=\exp\big(c(\log d)(\log\log d)\big).
\end{equation}
\end{theorem}

By the naive height $H_0(\alpha)$ of an algebraic number $\alpha$, we mean the largest absolute value of the coefficients of its irreducible polynomial $P_\alpha$ in $\bZ[x]$, while its degree $d(\alpha)$ is the degree of $P_\alpha$.

Fix $(\xi,\eta)\in\bR^2$ and, for each $X\ge 1$, define $\Delta(X)$ to be the minimum of the quantities
\[
 \delta(\ux):=\max\{|x_0\xi-x_1|,|x_0\eta-x_2|\}
\]
as $\ux=(x_0,x_1,x_2)$ runs through the points of $\bZ^3$ with $1\le x_0\le X$.  Then the condition that \eqref{intro:thm:eq1} has a non-zero integer solution for a given $X\ge 1$ is equivalent to asking that $\Delta(X)\le X^{-\lambda}$.  By a theorem of Dirichlet, we have $\Delta(X)\le X^{-1/2}$ for each $X\ge 1$.  We even have $\Delta(X)\le cX^{-1}$ with a constant $c>0$ that is independent of $X$ if $1,\xi,\eta$ are linearly dependent over $\bQ$.  However, if $\xi$ and $\eta$ are algebraic over $\bQ$ and if $1,\xi,\eta$ are linearly independent over $\bQ$, then Schmidt subspace theorem \cite[Ch.~VI,Theorem 1B]{Sc} implies that, for a given $\lambda>1/2$, the inequality
\[
 |x_0|^{2\lambda}\,|x_0\xi-x_1|\,|x_0\eta-x_2| \le 1
\]
has only finitely many solutions $(x_0,x_1,x_2)\in\bZ^3$ with $x_0\neq 0$.  This in turn implies that $\Delta(X)\ge c_\lambda X^{-\lambda}$ for each $X\ge 1$ with a constant $c_\lambda>0$.  Thus any point $(\xi,\eta)$ satisfying the hypotheses of the theorem has at least one transcendental coordinate.  However, if $\xi$ is algebraic over $\bQ$, then $f(\xi,y)$ is a non-zero polynomial of degree at most two in $y$ (because $f(x,y)$ is irreducible and depends on $y$), and therefore $\eta$ is also algebraic, a contradiction.  So, $\xi$ must be transcendental.  This proves the first part in the conclusion of the theorem.  To establish the measure of transcendence \eqref{intro:thm:eq2}, we follow Adamczewski and Bugeaud in \cite{AB} by using a quantitative version of Schmidt subspace theorem, namely that of Evertse from \cite{Ev}.  We recall the latter result in Section \ref{sec:qst} and postpone the proof of the theorem to Sections \ref{sec:minimal} and \ref{sec:proof}.

Let $\gamma=(1+\sqrt{5})/2\simeq 1.618$ denote the golden ratio.  In \cite{AB}, Adamczewski and Bugeaud consider the case where $(\xi,\eta)=(\xi,\xi^2)$ is a point on the parabola $y=x^2$.  For such a point, the condition that $1,\xi,\eta$ are linearly independent over $\bQ$ amounts to asking that $\xi$ is not rational nor quadratic over $\bQ$.  In that case, Davenport and Schmidt \cite{DSb} showed the existence of a constant $c>0$ such that $\Delta(X)\ge cX^{-1/\gamma}$ for arbitrarily large values of $X$. In \cite{RcubicI}, we proved that, conversely, there exist transcendental real numbers $\xi$, called \emph{extremal numbers}, for which the pair $(\xi,\eta)=(\xi,\xi^2)$ satisfies $\Delta(X)\le c'X^{-1/\gamma}$ for each $X\ge 1$, with another constant $c'>0$.  Such pairs thus satisfy the hypotheses of the theorem for any choice of $\lambda$ in $(1/2,1/\gamma)$.  Examples of extremal numbers include all real numbers whose continued fraction expansion is the Fibonacci word on two distinct positive integers \cite{Rnote}.  In \cite{BL}, Bugeaud and Laurent consider more generally the real numbers $\xi$ whose continued fraction expansion is a Sturmian word on two distinct positive integers.  When the slope of the Sturmian word has itself bounded partial quotients (like the slope $1/\gamma$ of the Fibonacci word), they determine an explicit and best possible value $\hlambda>1/2$ such that the pair $(\xi,\xi^2)$ satisfies the hypotheses of the theorem for each $\lambda\in(1/2,\hlambda)$.  For such Sturmian continued fractions $\xi$ and for extremal numbers $\xi$, Adamczewski and Bugeaud prove a measure of transcendence of the form
\[
  |\xi-\alpha|
   \ge H^{-w(d)} \quad \text{where} \quad w(d)=\exp\big(c(\log d)^2\cdot(\log\log d)^2\big).
\]
(see \cite[\S 5]{AB}).  Our main improvement is thus to remove the square on the term $\log(d)$.  However this is still not enough to conclude that $\xi$ is an S-number in the sense of Mahler, as one could expect, since this requires a measure of the form $|\xi-\alpha| \ge H^{-cd}$.

The numbers of Sturmian type introduced by A.~Po\"els in \cite{Po} include all Sturmian continued fractions mentioned above, and provide further examples of real numbers $\xi$ for which the point $(\xi,\xi^2)$ satisfies the hypotheses of our theorem.  So, our measure \eqref{intro:thm:eq2} applies to these numbers as well.

In general, if a polynomial $f\in\bQ[x,y]$ of degree $2$ admits at least one zero $(\xi,\eta)$ with $1,\xi,\eta$ linearly independent over $\bQ$ then $f$ is irreducible over $\bQ$ and its gradient does not vanish at the point $(\xi,\eta)$.  Thus the equation $f(x,y)=0$ defines a conic in $\bR^2$ with infinitely many points.  In \cite[Theorem 1.2]{Rconic}, we show that there are points $(\xi,\eta)$ on that curve which satisfy the hypotheses of the theorem for any choice of $\lambda$ in $(1/2,1/\gamma)$ (and none for any $\lambda>1/\gamma$). So, if $f\notin\bZ[x]$, then $\xi$ is transcendental and satisfies \eqref{intro:thm:eq2}.

We do not know if the theorem applies to irreducible polynomials $f\in\bQ[x,y]$ of degree $\deg(f)>2$.  We do not even know if  such a polynomial could have a zero $(\xi,\eta)$ which fulfills the hypotheses of the theorem.  In particular, we wonder if there are such ``singular'' points on the plane cubic $y=x^3$ and if so, what is the supremum of the corresponding values of $\lambda$.

\section{The quantitative subspace theorem}
\label{sec:qst}

To state the notion of height used by Evertse in \cite{Ev}, let $\Qbar$ denote the algebraic closure of $\bQ$ in $\bC$, let $K\subset\Qbar$ be a subfield of finite degree $d$ over $\bQ$, and let $n\ge 2$ be an integer.  For each place $v$ of $K$, we denote by $K_v$ the completion of $K$ at $v$, by $d_v=[K_v:\bQ_v]$ the local degree of $K$ at $v$,  and by $|\ |_v$ the absolute value on $K_v$ which extends the usual absolute value on $\bQ$ if $v$ is archimedean or the usual $p$-adic absolute value on $\bQ$ (with $|p|_v=p^{-1}$) if $v$ lies above a prime number $p$.  Then the absolute Weil height of a non-zero point $\ua=(a_1,\dots,a_n)\in K^n$ is
\[
 H(\ua) =
   \prod_{v|\infty}(|a_1|_v^2+\cdots+|a_n|_v^2)^{d_v/(2d)}
   \prod_{v\nmid\infty}\max\{|a_1|_v,\dots,|a_n|_v\}^{d_v/d}
\]
where the first product runs over the archimedean places of $K$ and the second one over all remaining places of $K$.  This height is called absolute because, for a given non-zero $\ua\in\Qbar^n$, it is independent of the choice of a number field $K\subset\Qbar$ such that $\ua\in K^n$.  Moreover it is projective in the sense that $H(\ua)=H(c\ua)$ for any $c\in\Qbar\setminus\{0\}$.

For any non-zero linear form $\ell(\ux)=a_1x_1+\cdots+a_nx_n\in\Qbar x_1+\cdots+\Qbar x_n$, we define the degree of $\ell$ to be the degree of the extension of $\bQ$ generated by all quotients $a_i/a_j$ with $a_j\neq 0$, and its height to be $H(a_1,\dots,a_n)$.  Then \cite[Corollary]{Ev} reads as follows.

\begin{theorem}[Evertse, 1996]
\label{thm:Evertse}
Let $n\ge 2$ be an integer, let $\ell_1,\dots,\ell_n$ be $n$ linearly independent linear forms in $n$ variables with coefficients in $\Qbar$, let $D$ be an upper bound for their degrees, and let $H$ be an upper bound for their heights.  Then, for every $\delta$ with $0<\delta<1$, there are proper linear subspaces $T_1,\dots,T_t$ of $\bQ^n$ with
\[
 t \le 2^{60n^2}\delta^{-7n}\log(4D)\cdot\log\log(4D)
\]
such that every non-zero point $\ux\in\bZ^n$ with $H(\ux)\ge H$ satisfying
\begin{equation}
 \label{eq:thm:Evertse}
 |\ell_1(\ux)\cdots\ell_n(\ux)| \le |\det(\ell_1,\dots,\ell_n)| H(\ux)^{-\delta}
\end{equation}
lies in $T_1\cup\cdots\cup T_t$.
\end{theorem}

Note that the precise statement of \cite[Corollary]{Ev} deals only with primitive 
points $\ux\in\bZ^n$, namely non-zero integer points whose coordinates are
relatively prime as a set.  However if a non-zero point $\ux\in\bZ^n$ 
satisfies \eqref{eq:thm:Evertse}, then the primitive points $\uy$ of which
it is an integer multiple are also solutions of \eqref{eq:thm:Evertse},
and so that restriction is not necessary.

\section{A sequence of minimal points}
\label{sec:minimal}

Let the notation and the hypotheses be as in the statement of the theorem.  The function $\Delta\colon[1,\infty)\to\bR$ attached to the pair $(\xi,\eta)$ is monotone decreasing to zero, and constant in each interval between two consecutive integers.  Let $X_1=1<X_2<X_3<\cdots$ be its points of discontinuity listed in increasing order, together with $1$.  For each index $i\ge 1$, we set $\Delta_i=\Delta(X_i)$ and choose a non-zero point $\ux_i=(x_{i,0},x_{i,1},x_{i,2})\in\bZ^3$ such that
\[
 x_{i,0}=X_i \et \delta(\ux_i) = \Delta_i.
\]
Following Davenport and Schmidt in \cite{DSa,DSb}, we say that $(\ux_i)_{i\ge 1}$ is a sequence of \emph{minimal points} for $(\xi,\eta)$.  In this section, we establish some of its properties starting with the most fundamental one.

Since $\Delta$ is constant on $[X_i,X_{i+1})$ with $\Delta(X)\le X^{-\lambda}$ for each large enough value of $X$, there exists $i_0\ge 2$ such that
\begin{equation}
 \label{eq:Delta}
 \Delta_i\le X_{i+1}^{-\lambda}
 \quad\text{for each}\quad i\ge i_0.
\end{equation}

\begin{lemma}
\label{lemma:W}
For each $i\ge 1$, the subspace $W_i=\langle\ux_i,\ux_{i+1}\rangle_\bQ$ of\/ $\bQ^3$ spanned by $\ux_i$ and $\ux_{i+1}$ has dimension $2$ and $\{\ux_i,\ux_{i+1}\}$ forms a basis of\/ $W_i\cap\bZ^3$.
\end{lemma}

\begin{proof}
The points $\ux_i$ and $\ux_{i+1}$ are primitive with $\ux_{i+1}\neq\pm\ux_i$.  So they span a subspace of $\bQ^3$ of dimension $2$.  For the second assertion, it suffices to adapt the argument in the proof of \cite[Lemma 2]{DSa}.
\end{proof}

For any basis $\{\ux,\uy\}$ of $W_i\cap\bZ^3$, the cross product $\ux\wedge\uy$ is a primitive element of $\bZ^3$ which, by the lemma, is equal to $\pm\ux_i\wedge\ux_{i+1}$.  Upon defining the height $H(W_i)$ of $W_i$ as the Euclidean norm of that vector, we obtain
\begin{equation}
\label{eq:heightW}
 H(W_i)=\|\ux_i\wedge\ux_{i+1}\|_2 \ll X_{i+1}\Delta_i \ll X_{i+1}^{1-\lambda}
\end{equation}
with implied constants that do not depend on $i$.

\begin{lemma}
\label{lemma:X}
Let $I$ denote the set of indices $i\ge 2$ such that $\ux_{i-1}$, $\ux_i$ and $\ux_{i+1}$ are linearly independent over $\bQ$.  Then $I$ is an infinite set.  For any pair of consecutive elements $i<j$ of $I$ (in the natural ordering inherited from $\bN$), we have $W_i\neq W_j$ and $X_j \le H(W_i)H(W_j)$.
\end{lemma}

\begin{proof}
The argument of \cite[Lemma 5]{DSb} shows that $I$ is infinite.  Let $i<j$ be consecutive elements $i<j$ of $I$.  We have $W_{i-1}\neq W_i=\cdots=W_{j-1}\neq W_j$, thus $W_i\neq W_j$.  Moreover the points $\ux_i\wedge\ux_{i+1} = \pm \ux_{j-1}\wedge\ux_j$ and $\ux_j\wedge\ux_{j+1}$ being orthogonal to $\ux_j$ and not parallel, their cross product is a non-zero integer multiple of $\ux_j$, and thus
\[
 X_j\le \|\ux_j\|_2\le \|\ux_i\wedge\ux_{i+1}\|_2\,\|\ux_j\wedge\ux_{j+1}\|_2 = H(W_i)H(W_j).
\]
\end{proof}

\begin{lemma}
\label{lemma:f}
For each $i\ge 1$, we have $X_{i+1}^\lambda\ll X_i$.
\end{lemma}

\begin{proof}
Let $\varphi\in\bQ[x_0,x_1,x_2]$ be the homogeneous quadratic form for which $f(x,y)=\varphi(1,x,y)$.  We claim that $\varphi(\ux_i)\neq 0$ for each sufficiently $i$.  If we take it for granted then, for each of those $i$, we have $1/c\le |\varphi(\ux_i)|$ where $c$ is a common denominator of the coefficients of $\varphi$.  As $\varphi(1,\xi,\eta)=0$, we also have $|\varphi(\ux_i)|\ll \|\ux_i\|_2 \Delta_i \ll X_iX_{i+1}^{-\lambda}$.  Combining the two estimates yields $X_{i+1}^\lambda\ll X_i$.

The claim is clear if $f$ has at most one zero in $\bQ^2$.  Otherwise, \cite[Lemma 2.4]{Rconic} shows that there exist $\mu\in\bQ^\times$ and $T\in\GL_3(\bQ)$ such that $\mu(\varphi\circ T)(x_0,x_1,x_2)=x_0x_2-x_1^2$.  Then $T^{-1}(1,\xi,\eta)$ is proportional to $\Theta=(1,\theta,\theta^2)$ for some $\theta\in\bR$ and, for each $i\ge 1$, the point $T^{-1}(\ux_i)$ is proportional to a primitive integral point $\uy_i=(y_{i,0},y_{i,1},y_{i,2})$ with $\|\uy_i\|_2\asymp X_i$ and $\|\uy_i\wedge\Theta\|_2\asymp \Delta_i$.  We now argue as Davenport and Schmidt in the proof of \cite[Lemma 2]{DSb}, omitting details. If $\varphi(\ux_i)=0$ for some $i$, then $\uy_i=\pm(m^2,mn,n^2)$ for some coprime integers $m,n$ with $|m|\asymp X_i^{1/2}$ and $|m\theta-n|\ll X_{i+1}^{-\lambda}X_i^{-1/2}$.  However, if $i\ge 2$, then $\uy_{i-1}$ is not proportional to $\uy_i$ and so we have $my_{i-1,j+1}\neq ny_{i-1,j}$ for some $j\in\{0,1\}$.  For that $j$, we find that $1\le |my_{i-1,j+1}-ny_{i-1,j}|\ll X_i^{1/2-\lambda}$ and so $i$ is bounded from above.
\end{proof}

\begin{lemma}
\label{lemma:main}
Let $\theta>(1-\lambda)/(2\lambda-1)$. Then, there exists an element $i_1$ of $I$ with $i_1\ge i_0$ such that, for any pair of consecutive elements $i<j$ of $I$ with $i\ge i_1$, we have $H(W_i)<H(W_j)$ and $X_{j+1}<X_{i+1}^\theta$.
\end{lemma}

\begin{proof}
Let $i<j$ be consecutive elements of $I$ with $i\ge i_0$.  If $H(W_j)\le H(W_i)$, then Lemma \ref{lemma:X} together with \eqref{eq:heightW} yields
\[
 X_{i+1}\le X_j\le H(W_i)H(W_j)\le H(W_i)^2\ll X_{i+1}^{2(1-\lambda)}.
\]
Since $2(1-\lambda)<1$, this cannot hold when $i$ is large enough.  For such $i$, we thus have $H(W_i)<H(W_j)$.  Combining Lemmas \ref{lemma:X} and \ref{lemma:f} with \eqref{eq:heightW}, we also find
\[
 X_{j+1}^\lambda \ll X_j\le H(W_i)H(W_j)\ll X_{i+1}^{1-\lambda}X_{j+1}^{1-\lambda},
\]
thus $X_{j+1}^{2\lambda-1}\ll X_{i+1}^{1-\lambda}$ and so $X_{j+1}<X_{i+1}^\theta$ if $i$ is large enough.
\end{proof}

\section{Proof of the main theorem}
\label{sec:proof}

In continuation with the preceding section, we suppose that 
$\xi$, $\eta$, $f(x,y)$ and $\lambda$ are as in the statement 
of the theorem.  We choose $\theta$ and $i_1$ as in Lemma \ref{lemma:X},
and list in increasing order $i_1<i_2<i_3<\cdots$ the 
elements of $I$ that follow $i_1$.  Then, $W_{i_1},W_{i_2},W_{i_3},\dots$
are subspaces of $\bQ^3$ of dimension $2$ with strictly increasing heights
and so they are pairwise distinct.  This will be important in what 
follows.  We also choose $\delta\in (0,1/2)$ such that 
\begin{equation}
 \label{eq:choix:delta}
 6\delta<2\lambda-1.
\end{equation}
All constants $C_1,C_2,\dots$ that appear below depend only,
in a simple way, on these data.  

Since $f(x,y)$ has degree $2$ and $\partial f/\partial y \neq 0$, 
we deduce from the linear independence of $1,\xi,\eta$ over $\bQ$
that $|\partial f/\partial y(\xi,\eta)| \neq 0$.
Thus, by the implicit function theorem, there exists $C_1>0$ and $C_2\ge 1$ such 
that, for any $\alpha\in\bC$ with $|\xi-\alpha|\le C_1$, we can find 
$\beta\in\bC$ satisfying 
\begin{equation}
 \label{eq:beta:analytic}
 f(\alpha,\beta)=0
 \et
 |\eta-\beta|\le C_2|\xi-\alpha|.   
\end{equation}

Fix integers $d\ge 3$ and $H\ge 2$ and an algebraic number $\alpha$ 
with degree $d(\alpha)\le d$ and naive height $H_0(\alpha)\le H$.  We need 
to provide a lower bound for 
\begin{equation}
 \label{eq:epsilon}
 \epsilon:=|\xi-\alpha|.
\end{equation}
Suppose first that $\epsilon\le C_1$ and choose $\beta\in\bC$ as in
\eqref{eq:beta:analytic}.  Since $f(x,y)$ is irreducible over $\bQ$ and depends on $y$,
it is relatively prime to the irreducible polynomial $P_\alpha(x)$
of $\alpha$ and so $\beta$ is a root of their resultant in $x$.
Thus $\beta$ is an algebraic number with
\begin{equation}
 \label{eq:beta:algebraic}
 d(\beta)\le 2d
 \et
 H_0(\beta)\le C_3^d H^2.
\end{equation}
 
Consider the linear forms with algebraic coefficients
\begin{equation}
 \label{eq:linform}
 \ell_1=x_0, \quad \ell_2=x_0\alpha-x_1 \et \ell_3=x_0\beta-x_2.
\end{equation}
By the above they have degree at most $2d$.  To estimate their 
heights, we note that, for any algebraic number $\gamma$, we have
\[
 H(1,\gamma)
  \le \sqrt{2}M(\gamma)^{1/d(\gamma)}
  \le C_4 H_0(\gamma)^{1/d(\gamma)}
  \le C_4 H_0(\gamma),
\]
where $M(\gamma)$ denotes the Mahler measure of $\gamma$.  Using
this crude estimate together with \eqref{eq:beta:algebraic}, 
we find that the linear forms \eqref{eq:linform} have heights 
at most
\[
 \max\{H(1,\alpha),H(1,\beta)\} 
  \le C_4\max\{H,C_3^d H^2\}
  \le H^{C_5 d},
\]
where the last inequality uses $H\ge 2$.  By Theorem \ref{thm:Evertse}
of Evertse,
there exist proper linear subspaces $T_1,\dots,T_t$ of $\bQ^3$ with
\begin{equation}
 \label{eq:t}
 t \le 2^{540}\delta^{-21}\log(8d)\cdot\log\log(8d)
   \le C_6(\log d)(\log\log d)
\end{equation}
such that every non-zero point $\ux=(x_0,x_1,x_2)\in\bZ^3$ with 
$H(\ux)\ge H^{C_5 d}$ satisfying
\[
 |x_0|\,|x_0\alpha-x_1|\,|x_0\beta-x_2| \le H(\ux)^{-\delta}
\]
lies in $T_1\cup\cdots\cup T_t$.

Let $\ell\ge 1$ be the smallest integer such that 
\begin{equation}
 \label{eq:ell}
 X_{i_\ell+1} \ge \max\{\, H^{C_5 d},\, (t+1)^{1/\delta},\, 4^{1/\delta} C_7\}
 \quad
 \text{where}
 \quad C_7=3+|\xi|+|\eta|.
\end{equation}
The subspaces $W_{i_\ell},\dots,W_{i_{\ell+t}}$ of $\bQ^3$
being all distinct, there is at least one index 
$j$ among $\{i_\ell,\dots,i_{\ell+t}\}$ for which 
$W_j\notin\{T_1,\dots,T_t\}$.  Fix such a choice of $j$.
Since $W_j$ has dimension $2$,
it is not contained in any $T_i$. Now, consider 
the points $\ux=a\/\ux_j+\ux_{j+1}$ with $a\in\bZ$.  
Since any two of them span $W_j$, each $T_i$ contains at most
of one these points.  Thus there is at least one choice of 
$a$ with $0\le a\le t$ for which $\ux\notin T_1\cup\cdots\cup T_t$.
Fix such a choice of $a$ and denote by $(x_0,x_1,x_2)$ the 
coordinates of the corresponding point $\ux=a\/\ux_j+\ux_{j+1}$.
Since $\{\ux_j,\ux_{j+1}\}$ is a basis of $W_j\cap\bZ^3$ 
(see Lemma \ref{lemma:W}), this point $\ux$ is primitive
and so $H(\ux)=\|\ux\|_2$ is its Euclidean norm.
This yields $H(\ux)\ge x_0 \ge X_{j+1}\ge X_{i_\ell+1}\ge H^{C_5 d}$ 
and thus, by the result of Evertse, we must have
\begin{equation}
 \label{eq:Evertse}
 |x_0|\,|x_0\alpha-x_1|\,|x_0\beta-x_2| > \|\ux\|_2^{-\delta}.
\end{equation}
Using \eqref{eq:Delta}, we find
\begin{align*}
 &|x_0|=x_0\le X_j+tX_{j+1}\le (t+1)X_{j+1}\\
 &\max\{|x_0\xi-x_1|,|x_0\eta-x_2|\}
 \le \Delta_j+t\Delta_{j+1}
 \le (t+1)\Delta_j
 \le (t+1)X_{j+1}^{-\lambda}.
\end{align*}
By \eqref{eq:ell}, we also have
$t+1 \le X_{i_\ell+1}^\delta\le X_{j+1}^\delta$, thus
these inequalities imply
\[
 |x_0|\le X_{j+1}^{1+\delta}
 \et
 \max\{|x_0\xi-x_1|,|x_0\eta-x_2|\}
 \le X_{j+1}^{-\lambda+\delta}.
\]
Since $\delta<1/2<\lambda$, this yields 
\[
 \|\ux\|_2
   \le |x_0|+|x_1|+|x_2|
   \le C_7 |x_0|
   \le C_7 X_{j+1}^2,
\]
where $C_7$ is as in \eqref{eq:ell}.
Using \eqref{eq:beta:analytic} and \eqref{eq:epsilon},
we also deduce that
\begin{align*}
 \max\{|x_0\alpha-x_1|,|x_0\beta-x_2|\}
 &\le |x_0|\max\{|\xi-\alpha|,|\eta-\beta|\} + X_{j+1}^{-\lambda+\delta}\\
 &\le C_2X_{j+1}^{1+\delta}\epsilon + X_{j+1}^{-\lambda+\delta}\\
 &\le 2\max\big\{C_2X_{j+1}^{1+\delta}\epsilon,\ X_{j+1}^{-\lambda+\delta}\big\}.
\end{align*}
Substituting these estimates into \eqref{eq:Evertse}, we obtain
\begin{equation}
 \label{eq:maineq}
 4 X_{j+1}^{1+\delta}
   \max\big\{C_2X_{j+1}^{1+\delta}\epsilon,\ 
             X_{j+1}^{-\lambda+\delta}\big\}^2
 \ge C_7^{-\delta} X_{j+1}^{-2\delta}.
\end{equation}
Suppose first that $C_2X_{j+1}^{1+\delta}\epsilon
< X_{j+1}^{-\lambda+\delta}$.  Then, after simplifications, 
we obtain, by virtue of the choice of $\delta$ in \eqref{eq:choix:delta},
\[
 4C_7^\delta 
   \ge X_{j+1}^{2\lambda-1-5\delta} 
   > X_{j+1}^\delta \ge X_{i_\ell+1}^\delta,
\]
in contradiction with \eqref{eq:ell}.  So, the inequality 
\eqref{eq:maineq} implies that
\begin{equation}
 \label{eq:est:epsilon}
 \epsilon 
   \ge 2^{-1}C_2^{-1}C_7^{-\delta/2}X_{j+1}^{-(3+5\delta)/2}
   \ge X_{j+1}^{-C_8}.
\end{equation}
We now use Lemma \ref{lemma:main} to estimate $X_{j+1}$ from 
above.  Since $j\in\{i_\ell,\dots,i_{\ell+t}\}$, we obtain
\[
 X_{j+1}\le X_{i_{\ell+t}+1} \le X_{i_\ell+1}^{\theta^t}.
\]
If $\ell=1$, then \eqref{eq:est:epsilon} yields
\[
 \epsilon
   \ge X_{i_1+1}^{-C_8\theta^t}
   \ge 2^{-C_9\theta^t}
   \ge H^{-C_9\theta^t}.
\]
Otherwise, by the choice of $\ell$ in \eqref{eq:ell}, we have
\[
 X_{i_{\ell-1}+1} 
  < \max\{\, H^{C_5 d},\, (t+1)^{1/\delta},\, 4^{1/\delta} C_7\}
  \le H^{C_{10} d}.  
\]
Since $X_{i_\ell+1}\le X_{i_{\ell-1}+1}^\theta$, a similar 
computation then gives
\[
  \epsilon
   \ge X_{i_{\ell-1}+1}^{-C_8\theta^{t+1}}
   \ge H^{-C_8C_{10} d\theta^{t+1}}.
\]
So, in both cases, we obtain $|\xi-\alpha|=\epsilon \ge H^{-w(d)}$
where
\[
 w(d)=C_{11}d\,\theta^t \le \exp\big(C_{12}(\log d)(\log\log d)\big),
\]
using the upper bound for $t$ in \eqref{eq:t}.  Finally, in the 
case where $\epsilon\ge C_1$, this remains true at the expense
of replacing $C_{12}$ by a larger constant if necessary.

%
%


\begin{thebibliography}{99}

%
\bibitem{AB}
  B.~Adamczewski and Y.~Bugeaud,
  Mesures de transcendance et aspects quantitatifs de la m\'ethode de Thue-Siegel-Roth-Schmidt,
  \textit{Proc.\ London Math.\ Soc.}\ \textbf{101} (2010), 1--26.
%
\bibitem{BL}
  Y.~Bugeaud, M.~Laurent,
  Exponents of Diophantine approximation and Sturmian continued fractions,
  \textit{Ann.\ Inst.\ Fourier (Grenoble)} \textbf{55} (2005),
  773-804.
%
\bibitem{DSa}
  H.~Davenport, W.~M.~Schmidt,
  Approximation to real numbers by quadratic irrationals,
  \textit{Acta Arith.\ }\textbf{13} (1967),
  169-176.
%
\bibitem{DSb}
  H.~Davenport, W.~M.~Schmidt,
  Approximation to real numbers by algebraic integers,
  \textit{Acta Arith.\ }\textbf{15} (1969),
  393--416.
%
\bibitem{Ev}
  J.-H.~Evertse,
  An improvement of the quantitative subspace theorem,
  \textit{Compositio Math.}\ \textbf{101} (1996), 225--311.
%
\bibitem{Po}
  A.~Po\"els,
  Exponents of diophantine approximation in dimension 2 for numbers of Sturmian type,
  preprint, 38 pages, arXiv:1711.07896 [math.NT].
%
%
\bibitem{Rnote}
  D.~Roy,
  Approximation simultan\'ee d'un nombre et de son carr\'e,
  \textit{C.\ R.\ Acad.\ Sci., Paris, ser.\ I} \textbf{336}
  (2003), 1--6.
%
\bibitem{RcubicI}
  D.~Roy, Approximation to real numbers by cubic algebraic integers I,
  \textit{Proc.\ London Math.\ Soc.\ }\textbf{88}
  (2004), 42--62.
%
\bibitem{Rconic}
D.~Roy, Rational approximation to real points on conics,
\textit{Ann.\ Inst.\ Fourier (Grenoble)} \textbf{63} (2013), 2331--2348.
%
%
\bibitem{Sc}
  W.~M.~Schmidt,
  \textit{Diophantine approximation}, Lecture Notes in Math.,
  vol.~785, Sprin\-ger-Verlag, 1980.
%
\end{thebibliography}
\end{document}